\documentclass[12pt]{article}
\usepackage{amsfonts}

\usepackage{graphicx}
\usepackage{amsmath}

\newtheorem{theorem}{Theorem}

\newtheorem{corollary}[theorem]{Corollary}

\newtheorem{definition}[theorem]{Definition}

\newenvironment{proof}[1][Proof]{\textbf{#1.} }{\ \rule{0.5em}{0.5em}}

\begin{document}

\title{{\LARGE Inverting weak random operators.
\footnote{This research was partially supported by the CONACYT.}}}
\author{
Jonathan Gutierrez-Pav\'{o}n
\thanks{Departamento de Matematicas, CINVESTAV-IPN, A. Postal 14-740, Mexico D.F. 07000, MEXICO. Email: jjgutierrez@math.cinvestav.mx}
and Carlos G. Pacheco
\thanks{Email: cpacheco@math.cinvestav.mx}}

\maketitle

\begin{abstract}
We analyze two weak random operators, initially motivated from
processes in random environment. Intuitively speaking these
operators are ill-defined, but using bilinear forms one can deal
with them in a rigorous way. This point of view can be found for
instance in the work Skorohod \cite{Skorohod}, and it remarkably
helps to carry out specific calculations. In this paper, we find
explicitly the inverse of such weak operators, by provinding the
forms of the so-called Green kernel. We show how this approach helps
to analyze the spectra of the operators. In addition, we provide the
existence of strong operators associated to our bilinear forms.
Important tools that we use are the Sturm-Liouville theory and the
stochastic calculus.
\end{abstract}

\\



\section{Introduction}\addcontentsline{toc}{chapter}{Introduction}

There are plenty of examples of probabilistic models where there is
an operator that resembles differential operator with coefficients
given in terms of the derivative of the Brownian motion. For
instance, the so-called stochastic heat equation or the random
Schr\"{o}dinger equations are well known cases studied in the
literature. In this paper we work with two examples of random
operators defined in a weak sense using bilinear forms. One of our
aims is to find the inverse, in a suitable sense, of such weak
random operators.

This kind of models are instances of the so-called Schr\"{o}dinger
operators with random potential. They have been important in
theoretical physics, in particular in the theory of disorder
systems, e.g. \cite{Papanicolaou}. The importance of these models is
well documented, see for instance \cite{Carmona}.

Let us mention two important examples. In \cite{Halperin} it is
consider the Schr\"{o}dinger operator with random potential
informally given by the expression
$$Lf(t)=-f^{\prime\prime}(t)+W^{\prime}(t)f(t),\ t\in [0,1],$$
where $W^{\prime}$ is white noise and it can be thought as the
derivative of the Brownian motion. One very first task is to give a
proper meaning of the operator $L$. As shown in \cite{Fuku} such
operator has a discrete spectrum given by a set of eigenvalues. It
turns out one can give expressions of the inverse operator, see
\cite{Pacheco}, which leads to spectral information.

In the context of random processes with random environment an
important model is the so-called Brox diffusion, see \cite{Brox},
amply studied in the literature. This process can be worked out as a
Markov process, and informally speaking the generator has the form
$$Lf(t)=\frac{1}{2}\left(-f^{\prime\prime}(t)+W^{\prime}(t)f^{\prime}(t)\right), \ t\in \mathbb{R}.$$
It turns out that one can analyze $L$ by finding its inverse, as
done in the companion paper \cite{GPP}, where a bounded version of
the Brox diffusion is studied. Moreover, there is a remarkable
similarity with an operator arising in the theory of random
matrices, see \cite{Ramirez}. Loosely speaking, such operator plays
the role of the infinite random matrix, and the spectrum helps to
charaterize the limiting eigenvalues of a random matrix

As it is traditionally thought, knowing spectral information of the
inverse helps to analyze the differential operator. As demonstrated
in \cite{GPP}, the inverse of $L$ helps to obtain spectral
information which eventually leads to information of the probability
density function. From a more theoretical point of view, one can see
that is possible to deal with the inverse in fairly friendly way,
without making use of machinary such as the theory of distributions.
This is so from well-known tools in the Sturm-Liouville theory and
the stochastic calculus.


In this paper the two operators that we consider are given
informally by the expressions:
$$(Lf)(t):=f''(t)-W(t)f'(t)-W'(t)f(t),$$
and
$$(Lf)(t):=\frac{f''(t)}{2}-\frac{W'(t)f'(t)}{2}.$$

In order to make sense of the term $W'$, we will define these
operators in a weak sense using the inner product. In that way we
can make sense of the term $\displaystyle\int_{a}^{b}W'(t)h(t)dt$ by
rewriting it as
\begin{equation} \label{cuadrado}
\displaystyle\int_{a}^{b}h(t)dW(t).
\end{equation}

After specifying the domains, our goal is to find the inverse of
these two operators defined in weak sense. This inverse operator is
called the Green operator. In the classical Sturm-Liouville theory,
to tackle this problem one should consider the solutions of the
homogeneous problem $Lf=0$. Here we will also consider the solutions
of the homogeneous equation but in a weak sense, again using the
inner product. It turns out that the homogeneous solutions are
explicit functions of the Brownian motions.

\vspace{0.2cm}

\noindent We start in the comming Section 2 with some preliminaries,
where we present the concept of a weak operator. In Section 2.1 we
also mention some ideas on strong operators associated to bilinear
forms. Then in Section 3.2 we deal with the first weak operator and
find explicitly the solutions of $Lf=0$. These solutions will help
to construct the green operator associated. In Section 3.1 we
mention how to find the strong operator associated to the weak
random operator. In a simular fashion, in Section 4 we work with the
second weak operator, and we also find explicit solutions of the
homogeneous equation using approximations of the Brownian motion.
\vspace{0.5cm}

\section{Preliminaries } \addcontentsline{toc}{chapter}{Introduction}

\vspace{0.2cm}

\noindent We will work with two weak random operators whose domain
are functions defined on an interval $[a,b]$. More precisely, the
domain is the set of functions $f \in L_{2}[a,b]$ absolutely
continuous that satisfies the Dirichlet conditions $f(a)=0=f(b)$.
Our first goal is to give the proper definitions of the operators
that we work using bilinear forms. Next we find solutions for the
homogeneous equation which eventually will lead to the inverse
operator.

\vspace{0.2cm}

\noindent The first operator that we consider has the following
formal expression:

$$(Lf)(t)=f''(t)-W(t)f'(t)-W'(t)f(t).$$

\vspace{0.2cm}

\noindent where $W:= \{W(t): t \in [a,b]  \}$ is a Brownian motion,
and $W'$ denotes its derivative, sometimes called the white noise.

\vspace{0.2cm}

\noindent The second operator that we consider can be expressed as
follows:

$$(Lf)(t)=\frac{f''(t)}{2}-\frac{W'(t)f'(t)}{2}.$$

\vspace{0.2cm}

\noindent A natural space to work with these operators is the
Hilbert space $L_{2}[a,b]$ with its inner product
$$\langle f,g \rangle=\int_{a}^{b}f(x)g(x)dx.$$
To define properly the domain of our operators we need to introduce
the following Sobolev space:
$$H_{1}:= \{h \in L_{2}[a,b] : h \mbox{ is absolutely continuous}, h(a)=h(b)=0   \}.$$
\vspace{0.2cm}

\noindent Note that $H_{1}$ is Hilbert space with the norm
\begin{equation} \label{norma sobolev}
\|f \|_{1}^{2}:= \displaystyle \int_{a}^{b}[f(x)]^{2}dx
+\int_{a}^{b}[f'(x)]^{2}dx,
\end{equation}
and the corresponding inner product.

\vspace{0.2cm}







The idea to define weak operators is to think of an operator $L$ by
describing its effect through the inner product, thus we will
propose a bilinear form. More specifically:
\begin{equation} \label{definicion del producto interno}
\langle Lf, h \rangle = \displaystyle\int_{a}^{b}Lf(t)h(t)dt, \; \;
\; \mbox{for all} \; \; f,h \in H_{1}.
\end{equation}
We take this point of view from the work of Anatolii Vladimirovich
Skorohod, see \cite{Skorohod}. \vspace{0.2cm}



\begin{definition}\label{WeakOperators}
Consider the mapping $\varepsilon (f,g)$ defined on a Hilbert space
with the following conditions:
\begin{enumerate}
    \item $\varepsilon ( \alpha_{1}f_{1}+\alpha_{2}f_{2},\beta_{1}g_{1}+\beta_{2}g_{2})
    = \displaystyle\sum_{i,j=1}^{2} \alpha_{i}\beta_{j}\varepsilon (f_{i},g_{j} )$,
    \item $\varepsilon ( f_{n},g_{n} )$ converges to
        $\varepsilon (f,g)$ in probability as $f_{n} \rightarrow f$ and $g_{n}\rightarrow g$.
\end{enumerate}
We say that $\varepsilon$ defines a weak random operator $L$,
through the expression $\langle Lf,g \rangle:=\varepsilon(f,g)$.
\end{definition}

On the other hand, as we mentioned in the Introduction, we need to
find the solutions of the homogeneous equation $Lf=0$. So, if $L$ is
a weak random operator we have the following definition of solving
$Lf=0$.
\begin{definition} \label{definicion solucion de la ecuacion homogenea}
We say that a stochastic process $\{ u(t): t \in [a,b] \}$ is a
solution of the equation $Lf=0$, if for all $h \in H_{1}$,
\begin{equation}
\langle Lu, h  \rangle= \varepsilon(u,h) =0 \; \; \mbox{almost
surely}.
\end{equation}
\end{definition}
It turns out that it is possible to find solutions of this problem
for the operators we consider.

\subsection{On strong operators}\label{definition closed}

In some cases it is possible to find an operator in strong sense
associated to the bilinear form. Generally speaking, such situation
occurs if $\varepsilon$ is what it is called a symmetric closed
lower semibounded bilinear form on a Hilbert space with inner
product $\langle \cdot , \cdot \rangle$. The reader can see
\cite{Selfadjoint} as a general reference, in particular Chapter 10.
The two examples that we will study are not symmetric, however the
two bilinear forms that we consider can be written as $\varepsilon=
\varepsilon_{1}+\varepsilon_{2}$, where $\varepsilon_{1}$ is
symmetric and $\varepsilon_{2}$ is coercive. We can use this
decomposition to find a strong operator associated to $\varepsilon$.

More precisely, on a linear subspace $D$ of the Hilbert space $H$
with norm $\|\cdot \|$, a symmetric bilinear form $\varepsilon_{1}$
is lower semibounded if there exists a constant $C$ such that
$\varepsilon_{1}(f,f)\geq C \|f \|^{2}$ for all $f \in D$. It is
also said that $\varepsilon_{1}$ is closed if $D$ is complete with
the norm
\begin{equation}\label{norma subt}
\| f\|_{\varepsilon_{1}}:= \left[ \varepsilon_{1}(f,f)+ (1-C)\|f
\|^{2} \right]^{\frac{1}{2}}.
\end{equation}
Then, we will be able to appeal to the Corollary 10.8 in
\cite{Selfadjoint} to show that $\varepsilon_{1}$ have associated a
self-adjoint operator, i.e. there exists an operator $L_{1}$ such
that $\varepsilon_{1}(f,g)= \langle L_{1}f,g \rangle$.

On the other hand, for the bilinear form $\varepsilon_{2}$ we will
use the Lax-Milgram theorem. To use this theorem we need to show
that $\varepsilon_{2}$ is bounded and coercive, i.e. if there are
two constants $C>0$ and $c> 0$ such that $|\varepsilon_{2}(f,f)|
\leq C \|f \|^{2}$ and $|\varepsilon_{2}(f,f)| \geq c \|f \|^{2}$,
respectively. If a bilinear form satisfies the previous properties
on the Hilbert space $H$ then there exists a operator $L_{2}$ such
that $\varepsilon_{2}(f,g)= \langle L_{2}f,g \rangle$, where
$\langle \cdot,\cdot \rangle$ is the inner product on $H$.

Then we obtain that
$$\varepsilon(f,g)=\varepsilon_{1}(f,g)+\varepsilon_{2}(f,g)= \langle L_{1}f,g \rangle+
\langle L_{2}f,g \rangle= \langle (L_{1}+L_{2})f,g \rangle.$$ The
previous equality shows that the bilinear form $\varepsilon$ has
associated the operator $L_{1}+L_{2}$.

Let us stress out that although it becomes feasible to give this
association of a strong operator, in this paper our main goal is to
study the weak random operator. This comes from the interest to
carry out calculations relying on the bilinear forms alone.


\section{With random potential and random coefficient} \addcontentsline{toc}{chapter}{Introduction}

\vspace{0.2cm}

\noindent In this section we consider the operator with the
following formal expression:
\begin{equation} \label{operador 2}
(Lf)(t)=f''(t)-W(t)f'(t)-W'(t)f(t).
\end{equation}

\vspace{0.2cm}

\noindent We can consider (\ref{operador 2}) in the following weak
sense, using the inner product
\begin{equation} \label{definicion debil 2}
\langle Lf,h  \rangle:= \displaystyle \int_{a}^{b} f''(t)h(t)dt -
\int_{a}^{b} f'(t)W(t)h(t)dt-\int_{a}^{b}f(t)h(t)dW(t).
\end{equation}

\vspace{0.2cm}

\noindent Now, we use integration by parts in the first term of
(\ref{definicion debil 2}) and the It$\hat{\mbox{o}}$'s formula in
the third term of (\ref{definicion debil 2}) to obtain the following
definition.

\begin{definition} \label{def 2}
For any pair $f,h \in H_{1}$, we define the bilinear form
$\varepsilon$ as
\begin{equation} \label{definicion debil 2 esta es la definitiva}
\varepsilon( f, h  ):= -\displaystyle \int_{a}^{b} f'(t)h'(t)dt +
\int_{a}^{b}f(t)h'(t)W(t)dt.
\end{equation}
\end{definition}

\noindent As we mentioned in the previous section, we consider this
bilinear form as a weak random operator $L$ through the expression
$\langle Lf,g \rangle:= \varepsilon(f,g)$. We do not go into
details, but it is possible to show that $L$, i.e. $\varepsilon$,
fits into Definition \ref{WeakOperators}.

Before we study the inverse operator of $L$, let us mention how we
can find an operator associated to the bilinear.

\subsection{A strong operator}\label{SectSO}

In order to find a strong operator associated to the bilinear form
$\varepsilon$ we carry on the following decomposition.

Notice that $\varepsilon= \varepsilon_{1}+\varepsilon_{2}$ where
$$\varepsilon_{1}(f,h):=
\displaystyle-\frac{1}{2}\int_{a}^{b}f'(t)h'(t)dt+\frac{1}{2}
\int_{a}^{b}f(t)h'(t)W(t)dt+\frac{1}{2}\int_{a}^{b}f'(t)h(t)W(t)dt,$$
and
$$\varepsilon_{2}(f,h):= \displaystyle-\frac{1}{2}\int_{a}^{b}f'(t)h'(t)dt +
 \frac{1}{2}\int_{a}^{b}f(t)h'(t)W(t)dt
-\frac{1}{2}\int_{a}^{b}f'(t)h(t)W(t)dt.$$

One can see that $\varepsilon_{1}$ is symmetric form on $H_{1}$ but
$\varepsilon_{2}$ is not symmetric on $H_{1}$.

Let us see that $\varepsilon_{1}$ is lower semibounded and closed
bilinear form. Let $M:=\max_{a\leq s \leq b} |W(s)|.$ Using $\mid ab
\mid \leq \displaystyle\frac{a^{2}+b^{2}}{2}$, then we have

\begin{eqnarray*}
  \varepsilon_{1}(f,f) &=&
  \displaystyle-\frac{1}{2}\int_{a}^{b}[f'(t)]^{2}dt+
\int_{a}^{b}f(t)f'(t)W(t)dt \\
   &\geq& \displaystyle-\frac{1}{2}\int_{a}^{b}[f'(t)]^{2}dt-M
\int_{a}^{b}\mid f(t)f'(t)\mid dt \\
   &\geq &
   \displaystyle-\frac{1}{2}\int_{a}^{b}[f'(t)]^{2}dt-\frac{M}{2}
\int_{a}^{b} [f(t)]^{2} dt -\frac{M}{2}
\int_{a}^{b} [f'(t)]^{2} dt \\
   &\geq& C \|f \|_{1},
\end{eqnarray*}
where $C$ is a constant that depends on $W$, and $\|\cdot \|_{1}$ is
defined in (\ref{norma sobolev}).

\noindent Then $\varepsilon_{1}$ is a semibounded form on $H_{1}$.
Let us now see that $\varepsilon_{1}$ is closed, this happens if the
Sobolev space $H_{1}$ is complete with the norm $\| \cdot
\|_{\varepsilon_{1}}$. Indeed, this is the case because $\| \cdot
\|_{\varepsilon_{1}}$ is equivalent to the norm $\| \cdot \|_{1}$ of
$H_{1}$. This implies that $\varepsilon_{1}$ is a closed form on
$H_{1}$. Therefore, using the Corollary 10.8 from
\cite{Selfadjoint}, there exists an operator $L_{1}$ in strong sense
with domain $H_{1}$ associated with the bilinear form
$\varepsilon_{1}$, i.e. $\varepsilon_{1} (f,g)= \langle L_{1}f,
g\rangle_{1}$, where $\langle \cdot , \cdot \rangle_{1}$ is the
inner product associated with the norm $\| \cdot \|_{1}$.
\vspace{0.2cm}

\noindent On the other hand, we use the Lax-Milgram theorem to show
that there exists an operator $L_{2}$ such that
$\varepsilon_{2}(f,g)= \langle L_{2}f,g\rangle_{1}$. To do that, we
show that $\varepsilon_{2}$ is bounded and coercive. We have
\begin{eqnarray*}
 |\varepsilon_{2}(f,f)| &=&
 \displaystyle\frac{1}{2}\int_{a}^{b}[f'(t)]^{2}dt \leq C \|f\|^{2}_{1}.
\end{eqnarray*}
The previous inequality shows that $\varepsilon_{2}$ is bounded. Let
us see why it is coercive. To do that, we use the Poincar\'e
inequality: $\|f \|\leq K \|f' \|$ for some constant $K >0$ and for
all $f\in H_{1}$. Then
\begin{eqnarray*}
 |\varepsilon_{2}(f,f)| &=&
 \displaystyle\frac{1}{2}\int_{a}^{b}[f'(t)]^{2}dt  \\
   &=& \frac{1}{4} \|f'\|^{2}+ \frac{1}{4} \|f'\|^{2} \\
   &\geq& \frac{1}{4} \|f'\|^{2} + \frac{1}{4 K} \|f\|^{2} \\
   &\geq& c \|f \|_{1}^{2}.
\end{eqnarray*}
Thus, $\varepsilon_{2}$ is coercive. Therefore, using the
Lax-Milgram theorem there exists an operator $L_{2}$ such that
$\varepsilon_{2}(f,g)= \langle L_{2}f,g \rangle_{1}$.

\noindent Then our bilinear form $\varepsilon$ is associated with
the operator $L_{1} + L_{2}$ with respect to the inner product $
\langle \cdot,\cdot \rangle_{1}$.

\subsection{The Green operator}
\noindent Now, we want to construct the Green operator associated to
the weak random operator $L$ from the Definition (\ref{definicion
debil 2 esta es la definitiva}). To this end, we need to find two
solutions linearly independent of the homogeneous equation.
Intuitively we have
$$f''(t)-W(t)f'(t)-W'(t)f(t)=0.$$

\vspace{0.2cm}

\noindent This equation can be rewritten as
$$f''(t)=[W(t)f(t)]'.$$

\vspace{0.2cm}

\noindent Moreover, integrating both side we arrive at
$$f'(t)=W(t)f(t)+ C, \; \; \; \mbox{where C is a constant}.$$

\vspace{0.2cm}

\noindent This equation is easy to solve, and we exhibit the
solutions in the following theorem. However, we rigourously verify
that the solutions satisfies the equation $Lf=0$.

\begin{theorem}
Two linearly independent solutions of the problem $Lf=0$ are the
following
\begin{equation} \label{definicion de u 2}
u(t):=\displaystyle\frac{e^{\int_{a}^{t}W(s)ds}
\displaystyle\int_{a}^{t}e^{-\int_{a}^{s}W(r)dr}ds}
{e^{\int_{a}^{b}W(s)ds}\displaystyle\int_{a}^{b}e^{-\int_{a}^{s}W(r)dr}ds},
\end{equation}

\begin{equation} \label{definicion de v 2}
v(t):=\displaystyle\frac{e^{\int_{a}^{t}W(s)ds}
\displaystyle\int_{t}^{b}e^{-\int_{a}^{s}W(r)dr}ds}
{e^{\int_{a}^{b}W(s)ds}\displaystyle\int_{a}^{b}e^{-\int_{a}^{s}W(r)dr}ds}.
\end{equation}
\vspace{0.2cm}

\noindent Furthermore, they satisfy $u(a)=0$, $u(b)=1$, $v(a)=1$ and
$v(b)=0$.

\end{theorem}

\begin{proof}
\noindent Let us verify that $u$ is solution. For $v$ is similar.
According to the Definition \ref{def 2} we need to show that
$\langle Lu, h \rangle =0$ for all $h \in H_{1}$, i.e.
\begin{equation} \label{gato}
-\displaystyle\int_{a}^{b}u'(t)h'(t)dt+\int_{a}^{b}u(t)h'(t)W(t)dt=0.
\end{equation}

\vspace{0.2cm}

\noindent Note that
\begin{equation} \label{derivada de u 2}
u'(t)= \displaystyle\frac{W(t)e^{\int_{a}^{t}W(s)ds}
\displaystyle\int_{a}^{t}e^{-\int_{a}^{s}W(r)dr}ds +1}
{e^{\int_{a}^{b}W(s)ds}\displaystyle\int_{a}^{b}e^{-\int_{a}^{s}W(r)dr}ds}.
\end{equation}

\vspace{0.2cm}

\noindent Substituting (\ref{derivada de u 2}) and the definition of
$u$ in (\ref{gato}), we end up with $\langle Lu, h  \rangle = 0$.
\end{proof}

\noindent Using previous two solutions, we construct the Green
operator. The following theorem shows the construction.

\begin{theorem}
Let $u,v$ two solution of $Lf=0$, such that $u(a)=0$ and $u(b)=1$
always, and $v(a)=1$ and $v(b)=0$ always. The stochastic Green
operator associated to L is given by

\begin{equation} \label{operador de green 2}
(Tf)(t):=\displaystyle \int_{a}^{b}G(t,s)f(s)ds,
\end{equation}

\vspace{0.4cm}

\noindent where
\begin{center}
\noindent $G(t,s):=   \left\{%
\begin{array}{ll}
    \displaystyle\frac{u(t)v(s)}{\alpha(s)}, & \hbox{a $\leq$ s $\leq$ t $\leq$ b;} \\
    \displaystyle\frac{u(s)v(t)}{\alpha(s)}, & \hbox{a $\leq$ t $\leq$ s $\leq$ b.} \\
\end{array}%
\right.     $,
\end{center}

\vspace{0.2cm}

\noindent and $$\alpha(t):= u'(t)v(t)-v'(t)u(t).$$

\vspace{0.2cm}

\noindent This operator $T$ is the right inverse of the operator $L$
in the sense that for all $h \in H_{1}$
$$\varepsilon ( Tf,h )= \langle LTf, h \rangle= \langle f,
 h \rangle\; \; \; \mbox{almost surely}.$$
\end{theorem}

\vspace{0.2cm}

\begin{proof}
We want to proof that $\langle L(Tf),h  \rangle = \langle f, h
\rangle $. First note that
\begin{equation}
(Tf)(t)= u(t) \displaystyle \int_{a}^{t}\frac{v(s)f(s)}{\alpha(s)}ds
+  v(t)\int_{t}^{b}\frac{u(s)f(s)}{\alpha(s)}ds.
\end{equation}

\vspace{0.2cm}

\noindent Calculating the derivative of $(Tf)$ and simplifying yield
\begin{equation} \label{1}
\frac{d [(Tf)(t)]}{dt}= u'(t) \displaystyle
\int_{a}^{t}\frac{v(s)f(s)}{\alpha(s)}ds +
v'(t)\int_{t}^{b}\frac{u(s)f(s)}{\alpha(s)}ds.
\end{equation}
\vspace{0.2cm}

\noindent From the Definition \ref{definicion debil 2 esta es la
definitiva} we have
\begin{equation} \label{2}
\langle L(Tf), h  \rangle = -\displaystyle \int_{a}^{b}
(Tf)'(t)h'(t)dt + \int_{a}^{b}(Tf)(t)h'(t)W(t)dt.
\end{equation}

\vspace{0.2cm}

\noindent After plugging (\ref{1}) into (\ref{2}) one arrives at

\begin{eqnarray} \label{ecuacion grande 2}
\langle L(Tf),h  \rangle  &=& \nonumber
-\displaystyle\int_{a}^{b}u'(t)
 \left[\int_{a}^{t}\frac{v(s)f(s)}{\alpha(s)}ds
\right] h'(t)dt + \displaystyle\int_{a}^{b}u(t)
\left[\int_{a}^{t}\frac{v(s)f(s)}{\alpha(s)}ds \right] h'(t)W(t)dt \\
   &-&\displaystyle\int_{a}^{b}v'(t)
\left[\int_{t}^{b}\frac{u(s)f(s)}{\alpha(s)}ds \right] h'(t)dt+
\displaystyle\int_{a}^{b}v(t)
\left[\int_{t}^{b}\frac{u(s)f(s)}{\alpha(s)}ds \right] h'(t)W(t)dt.
\end{eqnarray}

\vspace{0.2cm}

\noindent Now, we add and subtract in (\ref{ecuacion grande 2}) the
following three terms:
\begin{equation}
\displaystyle \int_{a}^{b} \frac{u'(t)v(t)f(t)h(t)}{\alpha(t)}dt,
\end{equation}

\begin{equation}
\int_{a}^{b} \frac{u(t)v'(t)f(t)h(t)}{\alpha(t)}dt,
\end{equation}

\begin{equation}
\int_{a}^{b} \frac{u(t)v(t)f(t)h(t)W(t)}{\alpha(t)}dt.
\end{equation}

\vspace{0.2cm}

\noindent Hence, after calculations,
\begin{eqnarray} \label{ecuacion 2 donde usamos que Lu=0 y Lv=0}
   \langle L(Tf),h  \rangle  \nonumber &=&
  -\displaystyle\int_{a}^{b}u'(t) \left[h(t)\int_{a}^{t}\frac{v(s)f(s)}{\alpha(s)}ds
\right]' dt + \displaystyle\int_{a}^{b}u(t)
\left[h(t)\int_{a}^{t}\frac{v(s)f(s)}{\alpha(s)}ds \right]' W(t)dt\\
   &-& \displaystyle\int_{a}^{b}v'(t) \left[h(t)\int_{t}^{b} \frac{u(s)f(s)}{\alpha(s)}ds
\right]' dt+\displaystyle\int_{a}^{b}v(t)
\left[h(t)\int_{t}^{b} \frac{u(s)f(s)}{\alpha(s)}ds \right]' W(t)dt\\
   &+&
   \int_{a}^{b} \frac{u'(t)v(t)f(t)h(t)}{\alpha(t)}dt -
   \int_{a}^{b} \frac{v'(t)u(t)f(t)h(t)}{\alpha(t)}dt.
   \nonumber
\end{eqnarray}

\vspace{0.2cm}

\noindent Using the fact that $Lu=0$ and $Lv=0$, we obtain the
result.
\end{proof}

One can see that almost surely $T$ is a compact operator, thus it
has a discrete spectrum. It means that the relation $Te=\lambda e$
holds for some eigenvalue $\lambda$ and eigenfunction $e$. After
taking $\langle LTe,h  \rangle$ we arrive to the equation
$$\langle Le,h \rangle= \langle e/\lambda, h\rangle.$$
Therefore
\begin{corollary}
Almost surely, the weak operator $L$ has a discrete spectrum in the
sense that for all $h\in H_{1}$, the relation
$$\langle Le,h \rangle= \langle \lambda e, h\rangle$$
holds for a contable number of $\lambda$ and $e\in H_{1}$.
\end{corollary}

\section{With random potential} \addcontentsline{toc}{chapter}{Introduction}

\vspace{0.3cm}

\noindent Informally speaking, we consider the following stochastic
operator
\begin{equation} \label{ecuacion de brox informal}
(Lf)(t)=\frac{f''(t)}{2}-\frac{W'(t)f'(t)}{2}.
\end{equation}

\vspace{0.2cm}

\noindent Taking into account equation (\ref{cuadrado}) and
(\ref{definicion del producto interno}), we define (\ref{ecuacion de
brox informal}) in the following weak sense

\begin{equation} \label{definicion debil 1}
\langle Lf, h  \rangle := \displaystyle
\int_{a}^{b}\frac{f''(t)h(t)}{2}dt -
\int_{a}^{b}\frac{f'(t)h(t)}{2}dW(t).
\end{equation}

\vspace{0.2cm}

\noindent We go an step further and instead of (\ref{definicion
debil 1}), we use integration by parts to obtain the following
definition.

\begin{definition}\label{definicion debil de un operador}
For any pair $f,h \in H_{1}$, we define the bilinear form
$\varepsilon$ as
\begin{equation} \label{definicion debil del operador sobre W22}
 \varepsilon(f, h ) := -\displaystyle \int_{a}^{b}\frac{f'(t)h'(t)}{2}dt -
\int_{a}^{b}\frac{f'(t)h(t)}{2}dW(t),
\end{equation}
and $L$ through $\langle Lf,g\rangle =\varepsilon (f,g)$.
\end{definition}
As we mentioned in previous section, one can check that $L$
satisfies the properties in Definition \ref{WeakOperators}.

\vspace{0.2cm}

\noindent
\subsection{A strong operator}
To talk about the strong operator associated to $\varepsilon$, in
this case
we consider the Sobolev space
$$W^{2,2}:= \{h \in L_{2}[a,b]: h',h'' \in L_{2}[a,b], h(a)=h(b)=0 \},$$
with the norm
\begin{equation} \label{normasobolev2}
\|f \|^{2}_{2}:= \displaystyle \int_{a}^{b}[f(x)]^{2}dx
+\int_{a}^{b}[f'(x)]^{2}dx + \int_{a}^{b}[f''(x)]^{2}dx.
\end{equation}
We want to prove the existence of an associated operator. Indeed,
using the It$\hat{\mbox{o}}$'s formula, we obtain for $f \in
W^{2,2}$,
$$\varepsilon(f,h) = \displaystyle \frac{-1}{2}\int_{a}^{b} f'(t)h'(t)dt+
  \int_{a}^{b}\frac{1}{2} \left[f''(t)h(t)+f'(t)h'(t) \right] W(t)dt.$$

Notice that $\varepsilon= \varepsilon_{1}+ \varepsilon_{2}$, where
$$\varepsilon_{1}(f,h):= -\displaystyle \int_{a}^{b}\frac{f'(t)h'(t)}{4}dt+
\int_{a}^{b}\frac{f''(t)h(t)W(t)}{4}dt+
\int_{a}^{b}\frac{f(t)h''(t)W(t)}{4}dt+
\int_{a}^{b}\frac{f'(t)h'(t)W(t)}{2}dt,$$ \noindent and
$$\varepsilon_{2}(f,h):=-\displaystyle
\int_{a}^{b}\frac{f'(t)h'(t)}{4}dt+
\displaystyle\frac{1}{4}\int_{a}^{b}f''(t)h(t)W(t)dt-
\frac{1}{4}\int_{a}^{b}f(t)h''(t)W(t)dt.$$

Let us see that $\varepsilon_{1}$ is symmetric lower semibounded and
closed bilinear form on $W^{2,2}.$ Take $f \in W^{2,2}$, and let
$M:=\max_{a\leq s \leq b} |W(s)|,$ then
\begin{eqnarray*}
  \varepsilon_{1}(f,f) &=& \displaystyle -\int_{a}^{b} \frac{1}{4}[f'(t)]^{2}dt+
  \int_{a}^{b}\frac{1}{2} \left[f''(t)f(t)+[f'(t)]^{2} \right] W(t)dt\\
   &\geq & \displaystyle -\int_{a}^{b} \frac{1}{4}[f'(t)]^{2}dt-
   \displaystyle\frac{M}{2} \left[
  \int_{a}^{b} |f''(t)f(t)| dt +\int_{a}^{b}[f'(t)]^{2}
  dt \right] \\
  &\geq & \displaystyle -\frac{1}{2}\int_{a}^{b} [f'(t)]^{2}dt-
   \displaystyle\frac{ M}{2} \left[
  \frac{1}{2}\int_{a}^{b} [f''(t)]^{2} dt + \frac{1}{2}\int_{a}^{b} [f(t)]^{2} dt
  +\int_{a}^{b}[f'(t)]^{2} dt \right] \\
  &\geq & C \left[\displaystyle \int_{a}^{b}[f(t)]^{2}dt
+\int_{a}^{b}[f'(t)]^{2}dt + \int_{a}^{b}[f''(t)]^{2}dt  \right],
\end{eqnarray*}
where $C$ is a constant depending $W$. Then we have that the
bilinear form $\varepsilon_{1}$ satisfies
$$\varepsilon_{1}(f,f)\geq C \|f \|^{2}_{2},$$
which concludes that $\varepsilon_{1}$ is a semibounded form on the
Sobolev space $W^{2,2}$.

Now we point out why $\varepsilon_{1}$ is closed. This is the case
because the norm $\| \cdot \|_{2}$, which makes $W^{2,2}$ complete,
is actually equivalent to the norm $\| \cdot \|_{\varepsilon_{1}}$,
as one can check it. This implies that $\varepsilon_{1}$ is a closed
form on $W^{2,2}$, and using the Corollary 10.8 from
\cite{Selfadjoint}, there exists an operator $L_{1}$ associated with
the bilinear form $\varepsilon_{1}$, that is, such that
$\varepsilon_{1}(f,g)= \langle L_{1}f, g\rangle_{2}$, where $\langle
\cdot, \cdot \rangle_{2}$ is the inner product associated with the
norm $\| \cdot \|_{2}$.

For $\varepsilon_{2}$ we apply the Lax-Milgram theorem. As in
Section \ref{SectSO}, one can see that $\varepsilon_{2}$ is bounded
and coercive. Then we obtain that there exists an operator $L_{2}$
such that $\varepsilon_{2}(f,h)=\langle L_{2}f,h\rangle_{2}$. Then
the bilinear form $\varepsilon$ is associated with the operator
$L_{1}+L_{2}$ using the inner product of $W^{2,2}.$


\subsection{The Green operator}
\noindent Our aim is to construct the so-called Green operator
associated to the weak random operator $L$ from the Definition
(\ref{definicion debil del operador sobre W22}). To do this task, we
notice that we need to find two linearly independent solutions of
the problem $Lf=0$.

\vspace{0.2cm}

\noindent It happens that the two linearly independent solutions
always exist; we will prove this fact later on. For the moment, let
us suppose that we already have the two solutions $u$ and $v$ of the
homogeneous equation. With these functions we are going to construct
an operator $T$, called the Green operator, which will be the
inverse operator of the weak random operator $L$.

\vspace{0.2cm}

\noindent The following theorem shows how to use the two solutions
of the homogeneous problem to construct $T$. We take the idea of
this constructions from the Sturm-Liouville theory.

\vspace{0.3cm}

\begin{theorem}
Let $u,v$ be solutions of $Lf=0$, such that $u(a)=0$ and $u(b)=1$
a.s., and $v(a)=1$ and $v(b)=0$ a.s. The stochastic Green operator
associated to the weak random operator $L$ is given by

\begin{equation} \label{operador de green}
(Tf)(t):=\displaystyle \int_{a}^{b}G(t,s)f(s)ds,
\end{equation}

\vspace{0.4cm}

\noindent where
\begin{center}
$G(t,s):=   \left\{%
\begin{array}{ll}
    \displaystyle\frac{2u(t)v(s)}{\alpha(s)}, & \hbox{a $\leq$ s $\leq$ t $\leq$ b;} \\
    \displaystyle\frac{2u(s)v(t)}{\alpha(s)}, & \hbox{a $\leq$ t $\leq$ s $\leq$ b.} \\
\end{array}%
\right.               $,
\end{center}
\vspace{0.2cm}

\noindent and $$\alpha(t):= u'(t)v(t)-v'(t)u(t).$$

\vspace{0.2cm}

\noindent The operator $T$ in (\ref{operador de green}) is the right
inverse of $L$ in the sense that for all $h \in H_{1}$
$$ \varepsilon (Tf,g) = \langle LTf, h  \rangle = \langle f, h  \rangle \; \; \mbox{almost
surely.} $$
\end{theorem}

\begin{proof}
\noindent Let $u,v$ be solutions of $Lf=0$, such that $u(a)=0$ and
$u(b)=1$ always, and that $v(a)=1$ and $v(b)=0$ always as well.

\vspace{0.2cm}

\noindent Note that
\begin{equation} \label{operador de green desarrollado}
(Tf)(t)= 2u(t)\displaystyle\int_{a}^{t}\frac{v(s)f(s)}{\alpha(s)}ds
+ 2v(t)\displaystyle\int_{t}^{b}\frac{u(s)f(s)}{\alpha(s)}ds.
\end{equation}

\vspace{0.2cm}

\noindent On calculating the derivative of (\ref{operador de green
desarrollado}) we obtain

\begin{equation} \label{derivada de Tf}
\frac{ d [(Tf)(t)] }{dt}= 2u'(t) \displaystyle\int_{a}^{t}
\frac{v(s)f(s)}{\alpha(s)}ds +\frac{2u(t)v(t)f(t)}{\alpha(t)}+
2v'(t)\displaystyle\int_{t}^{b}
\frac{u(s)f(s)}{\alpha(s)}ds-\frac{2u(t)v(t)f(t)}{\alpha(t)}.
\end{equation}

\vspace{0.2cm}

\noindent Note that in the above expression the first and last term
are canceled. Now, by using Definition \ref{definicion debil de un
operador}

\begin{equation} \label{composicion LT}
\langle L(Tf),h  \rangle = \frac{-1}{2} \left[
\displaystyle\int_{a}^{b} (Tf(t))'h'(t)dt + \int_{a}^{b}
(Tf(t))'h(t)dW(t) \right].
\end{equation}

\vspace{0.2cm}

\noindent Inserting (\ref{derivada de Tf}) into (\ref{composicion
LT}) we arrive at

\begin{eqnarray} \label{ecuacion grande}
 \langle L(Tf),h  \rangle    &   =&  \nonumber -\displaystyle\int_{a}^{b}u'(t)
 \left[\int_{a}^{t}\frac{v(s)f(s)}{\alpha(s)}ds
\right] h'(t)dt - \displaystyle\int_{a}^{b}u'(t)
\left[\int_{a}^{t}\frac{v(s)f(s)}{\alpha(s)}ds \right] h(t)dW(t) \\
 &-& \displaystyle\int_{a}^{b}v'(t) \left[\int_{t}^{b}\frac{u(s)f(s)}{\alpha(s)}ds \right]
h'(t)dt- \displaystyle\int_{a}^{b}v'(t)
\left[\int_{t}^{b}\frac{u(s)f(s)}{\alpha(s)}ds \right] h(t)dW(t).
\end{eqnarray}

\vspace{0.2cm}

\noindent Therefore, if we add and subtract in (\ref{ecuacion
grande}) the following two terms

$$\displaystyle\int_{a}^{b}\frac{u'(t)v(t)f(t)h(t)}{\alpha(t)}dt, \; \; \mbox{and} \;
\; \displaystyle\int_{a}^{b}\frac{u(t)v'(t)f(t)h(t)}{\alpha(t)}dt,$$
\vspace{0.2cm}

\noindent and we use the fact that
\begin{equation}
\left[  h(t) \int_{a}^{t}\frac{v(s)f(s))}{\alpha(s)}ds \right]'
=h'(t)\int_{a}^{t}\frac{v(s)f(s)}{\alpha(s)}ds+
h(t)\frac{v(t)f(t)}{\alpha(t)},
\end{equation}

\vspace{0.2cm}

\noindent we arrive at

\begin{eqnarray} \label{ecuacion donde usamos que Lu=0 y Lv=0}
   \langle L(Tf),h  \rangle  \nonumber &=&
  -\displaystyle\int_{a}^{b}u'(t) \left[h(t)\int_{a}^{t}\frac{v(s)f(s)}{\alpha(s)}ds
\right]' dt- \displaystyle\int_{a}^{b}u'(t)
\left[h(t)\int_{a}^{t}\frac{v(s)f(s)}{\alpha(s)}ds \right] dW(t)\\
   &-& \displaystyle\int_{a}^{b}v'(t) \left[h(t)\int_{t}^{b} \frac{u(s)f(s)}{\alpha(s)}ds
\right]' dt- \displaystyle\int_{a}^{b}v'(t)
\left[h(t)\int_{t}^{b} \frac{u(s)f(s)}{\alpha(s)}ds \right] dW(t)\\
   &+&
   \int_{a}^{b} \frac{u'(t)v(t)f(t)h(t)}{\alpha(t)}dt -
   \int_{a}^{b} \frac{v'(t)u(t)f(t)h(t)}{\alpha(t)}dt.
   \nonumber
\end{eqnarray}

\vspace{0.2cm}

\noindent Now, using the fact that $u$ and $v$ are solutions of
$Lf=0$ in the sense of Definition \ref{definicion solucion de la
ecuacion homogenea}, we see that only the last two terms in
(\ref{ecuacion donde usamos que Lu=0 y Lv=0}) survive. Thus we
finally arrive at
\begin{eqnarray*}
 \langle L(Tf),h  \rangle  &=& \int_{a}^{b} \frac{u'(t)v(t)f(t)h(t)}{\alpha(t)}dt-
  \int_{a}^{b} \frac{v'(t)u(t)f(t)h(t)}{\alpha(t)}dt \\
   &=& \int_{a}^{b} \left[\frac{u'(t)v(t)-u(t)v'(t)}{\alpha(t)}  \right] f(t)h(t)dt\\
   &=& \langle f,h  \rangle,
\end{eqnarray*}

\vspace{0.2cm}

\noindent where we have substitute the very definition of $\alpha$.
This concludes the proof.
\end{proof}

As previous section, since $T$ is compact, we have that
\begin{corollary}
The operator $L$ has a discrete spectrum in a weak sense.
\end{corollary}

\noindent Now, in order to use previous theorem, we need to find the
two solutions of $Lf=0$. We do so by using approximations of
Brownian motion.

\vspace{0.2cm}

\noindent First, to obtain intuitively such so solutions we consider
the followings approximations of $W$,

$$W_{n}(t):= n \left[ \left(\frac{j+1}{n}-t   \right)W\left(\frac{j}{n}  \right)
+\left( t- \frac{j}{n}  \right) W \left(\frac{j+1}{n}   \right)
\right]$$

\vspace{0.2cm}

\noindent where $t \in \left[\frac{j}{n}, \frac{j+1}{n}   \right]$,
and $j= 0, \pm 1, \pm2...$ Therefore, the random function $W_{n}$ is
almost everywhere differentiable.

\vspace{0.2cm}

\noindent Then the following equation is valid for almost every $t
\in [a,b]$

$$U_{n}''(t)=W_{n}'(t)U_{n}'(t).$$

\vspace{0.2cm}

\noindent We want to use $U_{n}(t)$ to find heuristically a solution
of

$$U''(t)=W'(t)U'(t).$$

\vspace{0.2cm}

\noindent We consider the change of variable $Z_{n}(t):=U_{n}'(t)$.
Then we obtain the new equation

\begin{equation} \label{ecuacion de aproximaciones del browniano}
Z_{n}'(t)=W_{n}'(t)Z_{n}(t).
\end{equation}

\vspace{0.2cm}

\noindent From the Corollary of Theorem 7.3 of \cite{ikedawatanabe},
we have that there exist a sequence $Z_{n}(t)$ of solutions of
(\ref{ecuacion de aproximaciones del browniano}) such that, with
probability one

\begin{equation}
Z_{n}(t) \rightarrow Z(t), \; \; \; \mbox{as} \; \; n  \rightarrow
\infty,
\end{equation}

\vspace{0.2cm}

\noindent where $Z(t)$ is solution of the stochastic differential
equation

\begin{equation} \label{ecuacion diferencial estocastica con Z}
dZ(t)=Z(t)dW(t).
\end{equation}

\vspace{0.2cm}

\noindent Then we obtain that with probability one

\begin{equation}
U_{n}'(t) \rightarrow Z(t), \; \; \; \mbox{as} \; \; n  \rightarrow
\infty.
\end{equation}

\vspace{0.2cm}

\noindent On the other hand, the equation (\ref{ecuacion diferencial
estocastica con Z}) has unique solution, and this solution is

\begin{equation}
Z(t)=e^{W(t)-\frac{t}{2}}.
\end{equation}

\vspace{0.2cm}

\noindent Hence

\begin{equation}
U_{n}'(t) \rightarrow e^{W(t)-\frac{t}{2}}, \; \; \; \mbox{as} \; \;
n  \rightarrow \infty.
\end{equation}

\vspace{0.2cm}

\noindent This implies that

\begin{equation}
U_{n}(t) \rightarrow \int_{a}^{t}e^{W(s)-\frac{s}{2}}ds, \; \; \;
\mbox{as} \; \; n  \rightarrow \infty.
\end{equation}

\vspace{0.2cm}

\noindent In the following theorem we verify rigourously that
$u(t):=C \cdot \int_{a}^{t}e^{W(s)-\frac{s}{2}}ds$ satisfies $Lu=0$,
where $C$ is an appropriate constant. We also consider other
solution
 $v$ that we need to construct the Green operator.

\begin{theorem} \label{teorema soluciones u y v}
Two linearly independent solutions of the problem $Lf=0$ are the
following integrals of Geometric Brownian motion
\begin{equation} \label{definicion de u}
u(t):=\frac{\displaystyle\int_{a}^{t}e^{W(s)-\frac{s}{2}}ds}
{\displaystyle\int_{a}^{b}e^{W(s)-\frac{s}{2}}ds}.
\end{equation}
\begin{equation} \label{definicion de v}
v(t):=\frac{\displaystyle\int_{t}^{b}e^{W(s)-\frac{s}{2}}ds}
{\displaystyle\int_{a}^{b}e^{W(s)-\frac{s}{2}}ds}.
\end{equation}
\vspace{0.2cm}

\noindent Furthermore, they satisfy $u(a)=0$, $u(b)=1$, $v(a)=1$ and
$v(b)=0$.
\end{theorem}

\vspace{0.4cm}

\begin{proof}
\noindent We verify that $u$ is solution of $Lf=0$. For $v$ is
similar. To do that, according to Definition \ref{definicion debil
de un operador} we want to show that $\langle Lu, h \rangle =0$ for
all $h \in H_{1}$, i.e.
\begin{equation} \label{ecuacion con u igual a cero}
\displaystyle\int_{a}^{b}\frac{u'(t)h'(t)}{2}dt +
\int_{a}^{b}\frac{u'(t)h(t)}{2}dW(t)=0.
\end{equation}

\vspace{0.2cm}

\noindent From the definition of $u$ in (\ref{definicion de u}), we
have
\begin{equation} \label{derivada de u}
u'(t)=\frac{e^{W(t)-\frac{t}{2}}}
{\displaystyle\int_{a}^{b}e^{W(s)-\frac{s}{2}}ds}.
\end{equation}

\vspace{0.2cm}

\noindent Let $\beta:= \left[
\displaystyle\int_{a}^{b}e^{W(s)-\frac{s}{2}}ds \right]^{-1}$, then

\begin{equation} \label{formula producto interno}
\langle Lu, h    \rangle = \frac{-\beta}{2}
 \left[
\displaystyle\int_{a}^{b}e^{W(t)-\frac{t}{2}}h'(t)dt +
\displaystyle\int_{a}^{b}e^{W(t)-\frac{t}{2}}h(t)dW(t)
 \right].
\end{equation}

\vspace{0.2cm}

\noindent On the other hand, applying the It$\hat{\mbox{o}}$'s
formula we obtain
\begin{equation} \label{aplicando formula de ito}
\displaystyle\int_{a}^{b}e^{W(s)-\frac{s}{2}}h(s)dW(s)=
h(b)e^{W(b)-\frac{b}{2}}-h(a)e^{W(a)-\frac{a}{2}} -
\displaystyle\int_{a}^{b}e^{W(s)-\frac{s}{2}}h'(s)ds.
\end{equation}

\vspace{0.2cm}

\noindent Substituting (\ref{aplicando formula de ito}) in
(\ref{formula producto interno}), and recalling that $h \in H_{1}$,
we arrive at

\begin{equation}
\langle Lu, h    \rangle = \frac{-\beta}{2}
 \left[
\displaystyle\int_{a}^{b}e^{W(t)-\frac{t}{2}}h'(t)dt +
h(b)e^{W(b)-\frac{b}{2}}-h(a)e^{W(a)-\frac{a}{2}} -
\displaystyle\int_{a}^{b}e^{W(t)-\frac{t}{2}}h'(t)dt
 \right]=0.
\end{equation}
\end{proof}

\textbf{Acknowledgements}. We thank all the corrections and comments
from both referees, it helped us tremendously to improve our paper.



\begin{thebibliography}{9}

\bibitem{Brox}  TH. Brox (1986). A one-dimensional diffusion process in a Wiener medium.
\textit{The Annals of Probability} \textbf{17}(4), pp. 1206--1218.

\bibitem{Carmona} R. Carmona and J. Lacroix (1990).
\emph{Spectral Theory of Random Schr$\ddot{o}$dinger Operators}.
Birkh$\ddot{\mbox{a}}$user.

\bibitem{Chenay} W. Cheney (2001).
\emph{Analysis for Applied Mathematics}. Springer.

\bibitem{Fuku}  M. Fukushima and S. Nakao (1977). On spectra of the Schr\"{o}dinger operator with a white Gaussian noise potential.
\textit{Zeitshrift f\"{u}r Wahrsheinlichkeitstheorie} \textbf{37},
pp. 267--274.


\bibitem{GPP} J.J. Gutierrez-Pavon and C.G. Pacheco (2017).
\emph{The killed Brox diffusion}. Submitted.

\bibitem{Halperin}  B.I. Halperin (1965). Green's functions for a particle in a one-dimensional random potential.
\textit{Physical review} \textbf{19}(1A), pp. 104--117.

\bibitem{ikedawatanabe} N. Ikeda and S. Watanabe (1981).
\emph{Stochastic Differential Equations and Diffusion Processes}.
North-Holland Mathematical Library.

\bibitem{Karatzas} I. Karatzas and S.E. Shreve (1991).
\emph{Brownian Motion and Stochastic Calculus}. Springer-Verlag.

\bibitem{Karlin taylor} S. Karlin and H.M. Taylor (1981).
\emph{A Second Course in Stochastic Processes}. Academic Press

\bibitem{Papanicolaou} G.C. Papanicolaou and S.R.S. Varadhan (1982).
\emph{Diffusions with Random Coefficients}. North-Holland Publishing
Company.

\bibitem{Pacheco} C.G. Pacheco (2016).
\emph{Green Kernel for a Random Schrodinger Operators}.
Communications in Contemporary Mathematics 18, no. 5.

\bibitem{Ramirez}  J. Ram\'{i}rez and B. Rider (2009). Diffusion at the random matrix hard edge.
\textit{Communications in Mathematical Physics} \textbf{288}(4), pp.
887--906.

\bibitem{Selfadjoint}  K. Schm$\ddot{\mbox{u}}$dgen (2012). \textit{Unbounded Self-adjoint Operators on Hilbert Space}, Springer.

\bibitem{Skorohod}  A.V. Skorohod (1984). \textit{Random Linear Operators},
D. Reidel Publishing Company.


\end{thebibliography}
\end{document}